\RequirePackage{fix-cm}
\documentclass[smallextended]{svjour3}       
\smartqed  
\usepackage{graphicx}
\usepackage {color}
\usepackage {tikz}

\smartqed  
\usepackage{graphicx}
\usepackage{amsmath}
\usepackage{amssymb}
\usepackage{mathtools}

\def\R{{\mathbb R}}

\def\argmin{{\rm argmin}}

\begin{document}

\title{Convergence Rate of Projected Subgradient Method with Time-varying Step-sizes}


\titlerunning{On Step-size of Subgradient Method}        

\author{Zhihan Zhu \and Yanhao Zhang \and Yong Xia 
}

\institute{Zhihan Zhu \and Yanhao Zhang \and Yong Xia (corresponding author) \at School of Mathematical Sciences, Beihang University, Beijing 100191, People's Republic of China.
	\email{\{zhihanzhu,\ yanhaozhang,\ yxia\}@buaa.edu.cn}
	  }     

\date{Received: date / Accepted: date}
\maketitle

\begin{abstract}
We establish the optimal ergodic convergence rate for the classical projected subgradient method with  time-varying step-sizes. This convergence rate remains the same even if we slightly increase the weight of the most recent points, thereby relaxing the ergodic sense.
\end{abstract}

\section{Introduction}
Consider the nonsmooth convex optimization problem
\begin{equation*}
	x^*\in\argmin_{x\in\mathcal{X}}f(x),
\end{equation*}
where $\mathcal{X}\subset\R^n$ is a compact convex set that is contained within the Euclidean ball $B(x^*,R)$, and $f$ is (possibly nonsmooth) convex and Lipschitz on $\mathcal{X}$, i.e., there is an $L>0$ such that $\Vert g\Vert\le L$ for any $g\in\partial f(x)\neq\emptyset$ and $x\in \mathcal{X}$.

The classical projected subgradient method (PSG) iterates the following equations for $t\ge 1$:
\begin{equation*}
\left\{	\begin{aligned}
		y_{t+1}&=x_t-\eta_t g_t, ~\text{where } g_t\in\partial f(x_t),\\
		x_{t+1}&=  \argmin_{x\in\mathcal{X}}\Vert x-y_{t+1} \Vert.	\end{aligned} \right.
\end{equation*}
Utilizing  the following constant step size
	\begin{equation}
		\eta_s\equiv\frac{R}{L\sqrt{t}}, ~s=1,\cdots,t,  \label{size1}
	\end{equation}	
PSG achieves an optimal ergodic convergence rate (see, for example, \cite{nesterov2018lectures,bubeck2015convex,MITlecture}) expressed as \footnote{ $f ((\sum_{s=1}^t x_s)/t)$ can be replaced with $\min_{s=1,\cdots,t}f(x_s)$ based on similar analysis.}
\begin{equation}
		f\left(\frac{\sum_{s=1}^t x_s}{t}\right)-f(x^*)\le \frac{RL}{\sqrt{t}}. \label{rate1}
\end{equation}

A more practical approach is to use a step-size that monotonically decreases towards $0$.
One such time-varying step-size, motivated by \eqref{size1} and suggested in \cite{nesterov2018lectures,bubeck2015convex}, is expressed as follows:
\begin{equation}
	\eta_s=\frac{R}{L\sqrt{s}},~s=1,\cdots,t.\label{size2}
\end{equation}
However, using this step size results in sub-optimal ergodic convergence rate \cite{nesterov2018lectures,bubeck2015convex}, as it adds an additional $\log(t)$ factor compared to the right-hand side of \eqref{rate1}.

The contribution of this note is to demonstrate that PSG with the time-varying step-size \eqref{size2} indeed achieves the following optimal convergence rate.
\begin{theorem} \label{main}
PSG with the time-varying step-size \eqref{size2} satisfies
 \begin{equation*}
	f\left(\frac{\sum_{s=1}^t x_s}{t}\right)-f(x^*)\le \frac{3RL}{2\sqrt{t}}.
\end{equation*}
 \end{theorem}
In Section 2, we present a more generalized convergence analysis, allowing us to provide some insightful observations.

\section{Analysis}
Consider PSG with a general step-size $\eta_s$, which is
assumed to be positive and non-increasing. We can establish the following convergence rate.
\begin{theorem}\label{Convergence}
For any fixed $k\ge -1$,
PSG with a positive and non-increasing step-size sequence $\{\eta_s\}$ satisfies
	\begin{equation}
		f\left(\frac{\sum_{s=1}^t \frac{1}{\eta_s^k}x_s}{\sum_{s=1}^t \frac{1}{\eta_s^k}}\right)-f(x^*)\le \frac{\frac{R^2}{\eta_t^{k+1}}+ L^2 \sum_{s=1}^t\frac{1}{\eta_s^{k-1}}}{2\sum_{s=1}^t\frac{1}{\eta_s^k}}. \label{bound}
	\end{equation}
\end{theorem}	
\begin{proof}
	According to the definition of subgradient, we have
	\begin{eqnarray}
			f(x_s)-f(x^*)&\le& g_s^T(x_s-x^*) \nonumber \\
			&=& \frac{1}{\eta_s}(x_s-y_{s+1})^T(x_s-x^*) \nonumber\\
			&=& \frac{1}{2\eta_s}(\Vert x_s-y_{s+1} \Vert^2+\Vert x_s-x^* \Vert^2-\Vert y_{s+1}-x^* \Vert^2)\label{eq1}\\
			&=& \frac{1}{2\eta_s}(\Vert x_s-x^* \Vert^2-\Vert y_{s+1}-x^* \Vert^2)+\frac{\eta_s}{2}\Vert g_s \Vert^2 \nonumber\\
&\le& \frac{1}{2\eta_s}(\Vert x_s-x^* \Vert^2-\Vert x_{s+1}-x^* \Vert^2)+\frac{\eta_s}{2}L^2,\label{ineq1}
	\end{eqnarray}
where \eqref{eq1} follows from the elementary identity $2a^Tb=\Vert a\Vert^2+\Vert b\Vert^2-\Vert a-b\Vert^2$, and \eqref{ineq1} holds since $\Vert g_s \Vert \le L$ and
	\begin{equation*}
		\Vert y_{s+1}-x^* \Vert^2\ge \Vert x_{s+1}-x^* \Vert^2,
	\end{equation*}
which is implied by  the projection theorem.

Consequently, we have
	\begin{eqnarray}
&&\left(\sum_{s=1}^t\frac{1}{\eta_s^k}\right)\left(		f\left(\frac{\sum_{s=1}^t \frac{1}{\eta_s^k}x_s}{\sum_{s=1}^t \frac{1}{\eta_s^k}}\right)-f(x^*)\right) \nonumber \\
&\le&	\sum_{s=1}^t\frac{1}{\eta_s^k}(f(x_s)-f(x^*))  ~~~~({\rm since~}f{\rm~is~convex})\nonumber\\
&\le& \sum_{s=1}^t\frac{1}{2\eta_s^{k+1}}(\Vert x_s-x^* \Vert^2-\Vert x_{s+1}-x^* \Vert^2)+\sum_{s=1}^t\frac{1}{2\eta_s^{k-1}}L^2 \label{ineq2}\\
			&=& \frac{1}{2\eta_1^{k+1}}\Vert x_1-x^* \Vert^2+\sum_{s=2}^t(\frac{1}{2\eta_s^{k+1}}-\frac{1}{2\eta_{s-1}^{k+1}})\Vert x_s-x^* \Vert^2 \nonumber\\
&&-\frac{1}{2\eta_{t}^{k+1}}\Vert x_{t+1}-x^* \Vert^2+\sum_{s=1}^t\frac{1}{2\eta_s^{k-1}}L^2\nonumber\\
			&\le& \frac{R^2}{2\eta_1^{k+1}}+\frac{R^2}{2}\sum_{s=2}^t(\frac{1}{\eta_s^{k+1}}-\frac{1}{\eta_{s-1}^{k+1}})+
\sum_{s=1}^t\frac{1}{2\eta_s^{k-1}}L^2\label{ineq3}\\
			&=& \frac{R^2}{2\eta_t^{k+1}}+\sum_{s=1}^t\frac{1}{2\eta_s^{k-1}}L^2,\nonumber
		 \end{eqnarray}
where \eqref{ineq2} follows from \eqref{ineq1}, and \eqref{ineq3} holds since $1/\eta_s^{k+1}-1/\eta_{s-1}^{k+1}\ge0$ when $k\ge -1$. The proof is complete.  \qed
\end{proof}	

\begin{remark}\label{r1}
By setting $k=-1$ in Theorem \ref{Convergence}, the upper bound on the right-hand side \eqref{bound} simplifies to
	\begin{equation*}
\frac{R^2+L^2\sum_{s=1}^t\eta_s^2}{2\sum_{s=1}^t\eta_s},
	\end{equation*}
	which is exactly the same as the result presented in \cite{nesterov2018lectures}. Then PSG with the time-varying step-size \eqref{size2} satisfies
\begin{equation}
	f\left(\frac{\sum_{s=1}^t \frac{1}{ \sqrt{s}}x_s}{\sum_{s=1}^t \frac{1}{\sqrt{s}}}\right)-f(x^*)\le \frac{2RL+RL\log t}{4(\sqrt{t+1}-1)}, \label{bound2}
\end{equation}
which is sub-optimal. Note that computing the weighted average of the iterates in the second half of the sequence yields the optimal convergence rate \cite[Corollary 3.2]{Lan}.
\end{remark}

%

\begin{remark}\label{r3}
By setting $k=0$ in Theorem \ref{Convergence},  we can immediately obtain the optimal convergence rate, as presented in Theorem \ref{main}.
\end{remark}

\begin{remark}
By setting any $k$ such that $k>-1$ in Theorem \ref{Convergence},  the convergence rate of  $f((\sum_{s=1}^t  x_s/\eta_s^k)/(\sum_{s=1}^t 1/ \eta_s^k))-f(x^*)$ will be $\mathcal{O}(1/\sqrt{t})$, without the presence of a $\log t$ factor. In comparison with the sub-optimal case \eqref{bound2}, when $k>0$,  the
 weighting scheme $(\sum_{s=1}^t  x_s/\eta_s^k)/(\sum_{s=1}^t 1/ \eta_s^k)$ assigns smaller weights to the initial points and larger weights to the most recent points. This new result can be referred to as the ``weak'' ergodic convergence rate.
\end{remark}

\begin{remark}\label{r4}
	We can apply the same proof techniques to extend the conclusion of weak ergodic convergence to mirror descent, Nesterov's dual averaging, and other schemes with time-varying step sizes for solving nonsmooth convex optimization, see \cite{nesterov2018lectures,bubeck2015convex}.

\end{remark}


\section*{Funding}
This research was supported by National Key R\&D Program of China under grant 2021YFA1003300, the National Natural Science Foundation of China under grant 12171021, and the Fundamental Research Funds for the Central Universities.

\section*{Data Availability}
The manuscript has no associated data.


\end{document}